\documentclass[11pt]{amsart}

\usepackage[top=3.5cm, bottom=3.5cm, left=2.5cm, right=2.5cm]{geometry}

\usepackage{amsfonts}
\usepackage{amssymb}
\usepackage{amsthm}
\usepackage{amsmath}
\usepackage{hyperref}
\usepackage{tikz}
\usepackage{pgf}
\usepackage{graphicx}
\usepackage[capitalise]{cleveref}

\usepackage{fancyhdr}
\fancypagestyle{firststyle}
{
   \fancyhf{}
   \fancyfoot[C]{\copyright \,The Mathematical Association of America}
}

\title{The Mahler conjecture in two dimensions via the probabilistic method}
\date{}
\author{Matthew C. H. Tointon}
\address{Insitut de Math\'ematiques, Universit\'e de Neuch\^atel, Rue Emile-Argand 11, CH-2000 Neuch\^atel, Switzerland}
\email{matthew.tointon@unine.ch}
\thanks{The author is supported by grant FN 200021\_163417/1 of the Swiss National Fund for scientific research. When the majority of this work was carried out he was supported by an EPSRC doctoral training grant awarded by the Department of Pure Mathematics and Mathematical Statistics, Cambridge, and a Leslie Wilson Scholarship from Magdalene College, Cambridge.}

\newtheorem{prop}{Proposition}[section]
\newtheorem{theorem}{Theorem}
\newtheorem{lemma}[prop]{Lemma}

\newtheorem*{conjecture}{Conjecture}

\theoremstyle{definition}

\theoremstyle{remark}
\newtheorem*{remark}{Remark}

\newcommand{\R}{\mathbb{R}}

\newcommand{\N}{\mathbb{N}}

\newcommand{\E}{\mathbb{E}}

\newcommand*{\vol}{\mathop{\textup{vol}}\nolimits}
\newcommand*{\area}{\mathop{\textup{area}}\nolimits}

\numberwithin{equation}{section}

\begin{document}
\maketitle
\thispagestyle{firststyle}
\begin{abstract}The \emph{Mahler volume} is, intuitively speaking, a measure of how ``round'' a centrally symmetric convex body is. In one direction this intuition is given weight by a result of Santal\'{o}, who in the 1940s showed that the Mahler volume is maximized, in a given dimension, by the unit sphere and its linear images, and only these. A counterpart to this result in the opposite direction is proposed by a conjecture, formulated by Kurt Mahler in the 1930s and still open in dimensions $4$ and greater, asserting that the Mahler volume should be minimized by a cuboid. In this article we present a seemingly new proof of the $2$-dimensional case of this conjecture via the probabilistic method. The central idea is to show that either deleting a random pair of edges from a centrally symmetric convex polygon, or deleting a random pair of vertices, reduces the Mahler volume with positive probability.
\end{abstract}

\section{Introduction}
A \emph{convex body} $A\subset\R^d$ is a compact convex set with non-empty interior; it is said to be \emph{centrally symmetric} if $x\in A$ precisely when $-x\in A$. This article concerns a long-standing and seemingly difficult question in convex geometry---the \emph{Mahler conjecture}---that attempts to give a rigorous version of the intuitively reasonable statement that cubes and octahedra are the ``least round'' centrally symmetric convex bodies.

Given a centrally symmetric convex body $A\subset\R^d$, the \emph{polar body} $A^\circ\subset\R^d$ is defined by
\begin{displaymath}
A^\circ=\{x\in\R^d:\langle x,a\rangle\le1\text{ for all }a\in A\}.
\end{displaymath}
The \emph{Mahler volume} $M(A)$ of $A$ is then defined to be
\begin{displaymath}
M(A):=\vol(A)\vol(A^\circ).
\end{displaymath}
Here, of course, $\vol(A)$ means the volume of $A$ (or, more formally, its Lebesgue measure). Note that the Mahler volume is invariant under invertible linear transformations of $\R^d$, since if $T$ is such a transformation then
\begin{equation}\label{eq:lin.tr}
T(A)^\circ=(T^*)^{-1}(A^\circ).
\end{equation}

The Mahler volume can be thought of as measuring how ``round'' $A$ is; the Mahler conjecture seeks to justify this line of thinking by showing that Euclidean balls and their linear images maximize the Mahler volume, whereas cubes and cross-polytopes and their linear images minimize it.

Write $B^d$ for the unit Euclidean ball in $\R^d$, write $Q^d=[-1,1]^d$ for the standard centred cube, and write $O_d=\{x\in\R^d:\|x\|_1\le1\}$ for the standard cross-polytope. Note that $O_d=(Q^d)^\circ$ and $Q^d=(O^d)^\circ$, so that $M(Q^d)=M(O_d)$. Precisely, then, the Mahler conjecture states that for any centrally symmetric convex body $A\subset\R^d$ we have $M(Q^d)\le M(A)\le M(B^d)$.

The second of these inequalities was proved by Santal\'{o} \cite{santalo} in the 1940s, having been previously proved by Blaschke \cite{blaschke} in the cases $d=2,3$.
%\begin{theorem}[Santal\'o]
%Let $A\subset\R^d$ be a centrally symmetric convex body. Then
%\begin{displaymath}
%M(A)\le M(B^d),
%\end{displaymath}
%with equality if and only if $A$ is an ellipsoid.
%\end{theorem}
%
The Mahler conjecture thus reduces to the lower bound, and in fact for the remainder of this article we refer to the lower bound only as the ``Mahler conjecture''.
\begin{conjecture}[Mahler \cite{mahler.con}]
Let $A\subset\R^d$ be a centrally symmetric convex body. Then
\[
M(A)\ge M(Q^d).
\]
\end{conjecture}
The $d=2$ case of this conjecture was proved by Mahler \cite{mahler}. Iriyeh and Shibata \cite{3d} have very recently released a proof the $d=3$ case. The conjecture seems to be open for $d\ge4$ \cite{tao.mahler}, although there are some partial results. For example, the conjecture is known to hold for certain specific classes of convex body (see \cite{lr}, for example), and the weaker inequality $M(A)\ge(\pi/4)^{d-1}M(Q^d)$ has been shown by Kuperberg \cite{kuperberg} to hold in full generality. Moreover, it is known that the cube is a \emph{local} minimizer of the Mahler volume \cite{nprz}, and more generally that so-called \emph{Hanner polytopes} (which have the same Mahler volume as the cube) are local minimizers \cite{kim}. For more information on progress on the Mahler conjecture the reader may consult the expository article of Tao contained in \cite{tao.mahler}, or the dissertation \cite{henze} of Henze, which also gives a detailed account of Mahler's proof of the $d=2$ case and a sketch proof of the upper bound $M(A)\le M(B^d)$.

In proving the $d=2$ case, Mahler actually proved the following theorem (see \cite[Lemma 2.9]{henze}).

\begin{theorem}[Mahler]\label{thm:Mahler.2d.orig}Let $n>2$, and let $A$ be a centrally symmetric convex polygon with $2n$ edges. Then there exists a centrally symmetric polygon $A'$ with $2(n-1)$ edges such that $M(A')<M(A)$.
\end{theorem}
This is sufficient since by limiting arguments we may always assume that the convex body $A$ in the Mahler conjecture is a polytope, and any centrally symmetric quadrilateral is a linear image of $Q^2$.

The purpose of this article is to present a new and, we hope, entertaining proof of \cref{thm:Mahler.2d.orig}, via a somewhat different method from that used by Mahler. In fact, we give a slight refinement of \cref{thm:Mahler.2d.orig}, which we now describe.

We may view a centrally symmetric convex polytope $A$ in $\R^d$ as the region bounded by a finite set $\mathcal{H}$ of hyperplanes, with central symmetry in particular implying that $H\in\mathcal{H}$ if and only if $-H\in\mathcal{H}$; let us call the minimal such $\mathcal H$ the \emph{hyperplane presentation} of $A$. Note that if the number of pairs of hyperplanes is greater than $d$ then there is at least one way to remove a pair of hyperplanes from the presentation in such a way that the remaining hyperplanes still define a centrally symmetric convex polytope. Alternatively, we may view a centrally symmetric convex polytope $A$ in $\R^d$ as the convex hull of a finite set $\mathcal{P}$ of points, with central symmetry in particular implying that $x\in\mathcal{P}$ if and only if $-x\in\mathcal{P}$; let us call the minimal such $\mathcal P$ the \emph{vertex presentation} of $A$. Note that if the number of pairs of points is greater than $d$ then there is at least one way to remove a pair of points from the presentation in such a way that the remaining points still define a centrally symmetric convex polytope.

Our refinement of \cref{thm:Mahler.2d.orig} is then as follows.
\begin{theorem}\label{thm:Mahler.2d}Let $n>2$, and let $A$ be a centrally symmetric convex polygon with $2n$ edges. Then there is either a way to remove an opposite pair of lines from the hyperplane presentation $A$, or a way to remove an opposite pair of points from the vertex presentation of $A$, in such a way that the Mahler volume decreases.
\end{theorem}

One of the most interesting aspects of our proof of Theorem \ref{thm:Mahler.2d} is that it uses the so-called \emph{probabilistic method}. The probabilistic method is a powerful tool, most often associated with combinatorics, that can be used to prove the existence of a certain mathematical object with certain given properties. The crucial insight of the probabilistic method is that in order to show that it is possible for a given object to have certain properties, it is sufficient to show that if the object is generated at random according to some probability distribution then it has the required properties with positive probability.
%Alternatively, if one wishes to show that there exists a version of the object in which a certain specific paramater takes a value at most $\alpha$, say, it is sufficient to prove prove that the \emph{expected} value of this parameter is at most $\alpha$ for some random way of generating the object.
The reader may consult the book \cite{alon-spencer} of Alon and Spencer for a more thorough description of the method and of its rich history and numerous applications.

\subsection*{Acknowledgements} I am grateful to two anonymous referees for careful readings of an earlier version of this article and a number of suggestions that have significantly improved the exposition.

\section{Outline of the argument}\label{sec:outline}
Given a centrally symmetric convex polytope $A$ in $\R^d$ we write $V(A)$ for the set of vertices of $A$, and $F(A)$ for the set of \emph{facets} of $A$, which is to say the set of faces of $A$ of dimension $d-1$. In the two-dimensional setting we write $E(A)$ instead of $F(A)$ to emphasize that facets are simply edges.

Each $s\in F(A)$ lies within a hyperplane in $\R^d$ appearing in the hyperplane presentation of $A$. Thus, there exists a unit vector $u(s)\in\R^d$ that is perpendicular to $s$, and some real number $c(s)>0$ such that $s\subset\{x\in\R^d:\langle x,u(s)\rangle=c(s)\}$, and we have
\begin{equation}
\label{eq:A.def}
A=\{x\in\R^d:\langle x,u(s)\rangle\le c(s)\text{ for all }s\in F(A)\}.
\end{equation}
It follows that $A^\circ$ is a centrally symmetric convex polytope in $\R^d$, and the set $V(A^\circ)$ of vertices of $A^\circ$ is precisely $\{v^\circ(s):s\in F(A)\}$, where for each $s\in F(A)$ we define the point $v^\circ(s)\in\R^d$ via
\begin{equation}\label{eq:dual.def}
v^\circ(s)=u(s)/c(s).
\end{equation}
For each vertex $v\in V(A^\circ)$ we denote by $e^\circ(v)$ the facet of $A$ that gave rise to $v$; thus $v=v^\circ(e^\circ(v))$. This notation is illustrated in Figure \ref{fig:def} for the two-dimensional case.

\begin{figure}[b]
\begin{center}
\begin{tikzpicture}
\draw (-0.05,0.05) -- (0.05,-0.05);
\draw (0.05,0.05) -- (-0.05,-0.05) node[below] {\textup{0}};

\draw (-1,2) -- node[above] {$s=e^\circ(x)$} (1,2) -- node[right] {$s'=e^\circ(x')$} (1,-2) -- (-1,-2) -- (-1,2);
\filldraw (-1,2) circle (2pt); %node[above left] {$y_1$};
\filldraw (1,2) circle (2pt); %node[above right] {$y_2$};
\filldraw (1,-2) circle (2pt); %node[below right] {$y_3$};
\filldraw (-1,-2) circle (2pt); %node[below left] {$y_4$};

\draw[dotted] (0,0) -- node[right] {$c(s)$} (0,2);
\draw[dotted] (0,0) -- node[above] {$c(s')$} (1,0);

\draw (0,-3) node {$A$};

\begin{scope}[shift={(1.3,0)}]

\draw (5-0.05,0.05) -- (5.05,-0.05);
\draw (5.05,0.05) -- (5-0.05,-0.05) node[below] {\textup{0}};

\draw (5,1) -- (7,0) -- (5,-1) -- (3,0) -- (5,1);
\filldraw (5,1) circle (2pt) node[above] {$x=v^\circ(s)$};
\filldraw (7,0) circle (2pt) node[right] {$x'=v^\circ(s')$};
\filldraw (5,-1) circle (2pt);
\filldraw (3,0) circle (2pt);

\draw (5,-2) node {$A^\circ$};
\end{scope}

\end{tikzpicture}
\caption{Illustration of the vertex and facet notation in two dimensions.}\label{fig:def}
\end{center}
\end{figure}
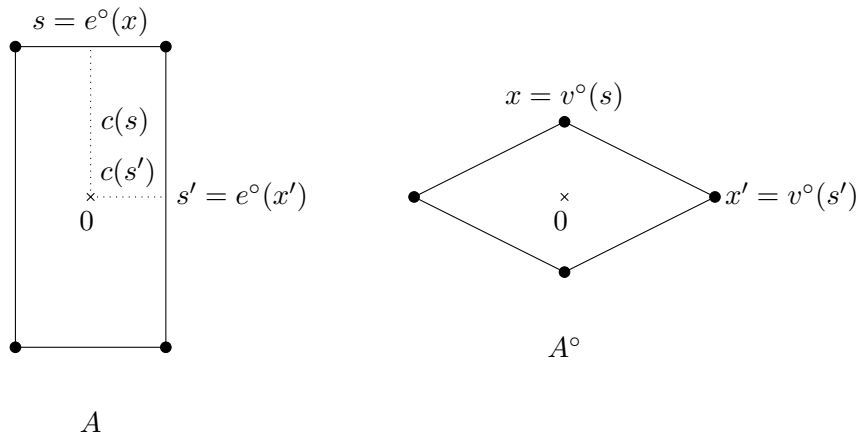

Throughout our proof of \cref{thm:Mahler.2d} we consider two types of parallelogram formed from the vertices of a given centrally symmetric convex polygon $A$, which we call \emph{type-1} parallelograms and \emph{type-2} parallelograms. We define the \emph{type-1 parallelogram based at the edge $e\in E(A)$}, written $P_1(e)$, to be the convex hull of the edges $\pm e$. We define the \emph{type-2 parallelogram based at the vertex $x\in V(A)$}, written $P_2(x)$, to be the convex hull of the four vertices adjacent to the vertices $\pm x$. Note that the definition of a type-2 parallelogram is degenerate when the ambient polygon is itself a parallelogram. These definitions are illustrated in Figure \ref{fig:2types}.  

\begin{figure}
\begin{center}
\begin{tikzpicture}

\draw (0.5,-2) -- node[below]{$e$} (-0.5,-2) -- (-1.16,-1.5) -- (-1.5,-0.5) -- (-1.5,0.5) -- (-1,1.5) -- (-0.5,2) -- node[above]{$-e$} (0.5,2) -- (1.16,1.5) -- (1.5,0.5) -- (1.5,-0.5) -- (1,-1.5) --cycle;

\draw (0,-2.5) node[below] {Type 1};

\filldraw[fill=black!20!white] (0.5,-2) -- (-0.5,-2) -- (-0.5,2) -- (0.5,2) -- cycle;
\draw (0,0) node {$P_1(e)$};

\filldraw (0.5,-2) circle (2pt);
\filldraw (-0.5,-2) circle (2pt);
\filldraw (-1.16,-1.5) circle (2pt);
\filldraw (-1.5,-0.5) circle (2pt);
\filldraw (-1.5,0.5) circle (2pt);
\filldraw (-1,1.5) circle (2pt);
\filldraw (-0.5,2) circle (2pt);
\filldraw (0.5,2) circle (2pt);
\filldraw (1.16,1.5) circle (2pt);
\filldraw (1.5,0.5) circle (2pt);
\filldraw (1.5,-0.5) circle (2pt);
\filldraw (1,-1.5) circle (2pt);

\begin{scope}[shift={(5,0)}]
\draw (0.5,-2) -- (-0.5,-2) node[below]{$x$} -- (-1.16,-1.5) -- (-1.5,-0.5) -- (-1.5,0.5) -- (-1,1.5) -- (-0.5,2) -- (0.5,2) node[above]{$-x$} -- (1.16,1.5) -- (1.5,0.5) -- (1.5,-0.5) -- (1,-1.5) --cycle;

\filldraw[fill=black!20!white] (0.5,-2) -- (-1.16,-1.5) -- (-0.5,2) -- (1.16,1.5) -- cycle;
\draw (0,0) node {$P_2(x)$};

\filldraw (0.5,-2) circle (2pt);
\filldraw (-0.5,-2) circle (2pt);
\filldraw (-1.16,-1.5) circle (2pt);
\filldraw (-1.5,-0.5) circle (2pt);
\filldraw (-1.5,0.5) circle (2pt);
\filldraw (-1,1.5) circle (2pt);
\filldraw (-0.5,2) circle (2pt);
\filldraw (0.5,2) circle (2pt);
\filldraw (1.16,1.5) circle (2pt);
\filldraw (1.5,0.5) circle (2pt);
\filldraw (1.5,-0.5) circle (2pt);
\filldraw (1,-1.5) circle (2pt);

\draw (0,-2.5) node[below] {Type 2};
\end{scope}

\end{tikzpicture}
\caption{The two types of parallelogram.}\label{fig:2types}
\end{center}
\end{figure}
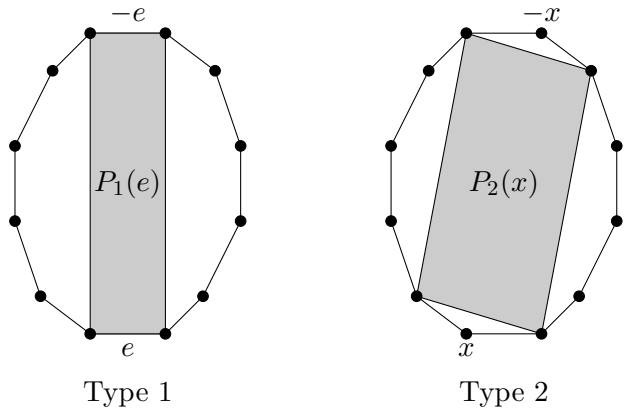

The first stage of our argument is to show that the type-1 parallelograms in a centrally symmetric convex polygon are smaller on average than the type-2 parallelograms, as follows.

\begin{prop}
\label{lem:quads.in.A}
Let $A$ be a centrally symmetric convex polygon in $\R^2$ with at least $6$ sides. Then
\begin{equation}\label{eq:cylcic.ineq}
\sum_{s\in E(A)}\area(P_1(s))\le\sum_{x\in V(A)}\area(P_2(x)),
\end{equation}
with equality if and only $A$ has exactly $6$ sides.
\end{prop}

The second stage of our argument shows how to pass from a certain comparison of type-1 parallelograms in $A$ with type-2 parallelograms in $A^\circ$ to the conclusion required by \cref{thm:Mahler.2d}, as follows.

\begin{prop}\label{prop:bures}
Let $A$ be a centrally symmetric convex polygon in $\R^2$ with at least $6$ sides, and let $s\in E(A)$. Let $A_s=\{x\in\R^d:\langle x,u(s')\rangle\le c(s')\text{ for }s'\in E(A)\backslash\{\pm s\}\}$ be the set obtained from $A$ by removing $\pm s$ from $E(A)$ in the presentation (\ref{eq:A.def}). Suppose that
\begin{equation}
\label{eq:Mahler.2d.j}
\frac{\area(P_1(s))}{\area(A)}\le\frac{\area(P_2(v^\circ(s)))}{\area(A^\circ)}.
\end{equation}
Then $M(A_s)<M(A)$.
\end{prop}

We prove these results in the next section using elementary geometry. For now, let us see how they combine to imply \cref{thm:Mahler.2d}.
\begin{proof}[Proof of \cref{thm:Mahler.2d}]
It is not difficult to check that if $A$ is a centrally symmetric convex body in $\R^d$ then $(A^\circ)^\circ=A$ (see \cite[Proposition 1.1]{tao.santalo}, for example). In the case $d=2$, note also that $A$ and $A^\circ$ are each centrally symmetric polygons with the same number of edges, and removing a pair of edges from the hyperplane presentation of one corresponds to removing a pair of vertices from the vertex presentation of the other. In proving \cref{thm:Mahler.2d}, therefore, we may interchange $A$ and $A^\circ$ without loss of generality. In particular, we may assume that
\begin{displaymath}
\E_{s\in E(A)}\frac{\area(P_1(s))}{\area(A)}\le\E_{t\in E(A^\circ)}\frac{\area(P_1(t))}{\area(A^\circ)}.
\end{displaymath}
Combining this with \cref{lem:quads.in.A} implies that
\begin{displaymath}
\E_{s\in E(A)}\frac{\area(P_1(s))}{\area(A)}\le\E_{s\in E(A)}\frac{\area(P_2(v^\circ(s)))}{\area(A^\circ)},
\end{displaymath}
and so---and this is where we apply the probabilistic method---there must exist some $s$ such that the inequality \eqref{eq:Mahler.2d.j} holds. \cref{prop:bures} implies that for this $s$ we have $M(A_s)<M(A)$, and so the theorem is proved.
\end{proof}

\begin{remark}
In dimension greater than $2$ we would no longer be able to interchange $A$ and $A^\circ$ in quite the same way, since they would not in general have the same number of facets or vertices. In three dimensions, for example, if $A$ were the cube and $A^\circ$ the cross-polytope then interchanging $A$ and $A^\circ$ would increase the number of facets of $A$. Indeed, deleting a pair of facets from the three-dimensional cross-polytope produces a linear image of the cube, and so the process of deleting pairs of facets would not even necessarily terminate if we allowed interchanges.
\end{remark}

\section{The details of the argument}\label{sec:proof}

In this section we prove Propositions \ref{lem:quads.in.A} and \ref{prop:bures}. Throughout, given points $x_1,\ldots,x_r\in\R^d$ we write $[x_1,\ldots,x_r]$ for their convex hull. Moreover, given a centrally symmetric convex polytope $A$, for each $x\in V(A)$ we write $x+1$ for the vertex neighbouring $x$ in a clockwise direction, and $x-1$ for the vertex neighbouring $x$ in an anticlockwise direction. More generally, for $k\in\N$ we define $x+k$ and $x-k$ recursively via $x+k=(x+(k-1))+1$ and $x-k=(x-(k-1))-1$. Note that if $A$ has $2n$ vertices then $x+n=-x$ for every $x\in V(A)$.
\begin{proof}[Proof of \cref{lem:quads.in.A}]
Equality is trivial when $A$ has $6$ sides, so we may assume it has $2n$ sides with $n\ge4$ and prove that the inequality \eqref{eq:cylcic.ineq} holds and is strict.

Given a vertex $x$ of $A$ write $H(x)$ for the hexagon $[x-1,x,x+1,x+n-1,x+n,x+n+1]$. Write $e(x)\in E(A)$ for the edge $[x-1,x]$, and note that each of $P_1(e(x))$ and $P_2(x)$ is a subset of $H(x)$. Carving $P_1(e(x))$ out from $H(x)$ leaves triangles $T_1(x)=[x,x+1,x+n-1]$ and $-T_1(x)$, as illustrated in Figure \ref{fig:carve}. Carving out $P_2(x)$ from $H(x)$ leaves the triangles $T_2(x)=[x-1,x,x+1]$ and $-T_2(x)$, as also illustrated in Figure \ref{fig:carve}.
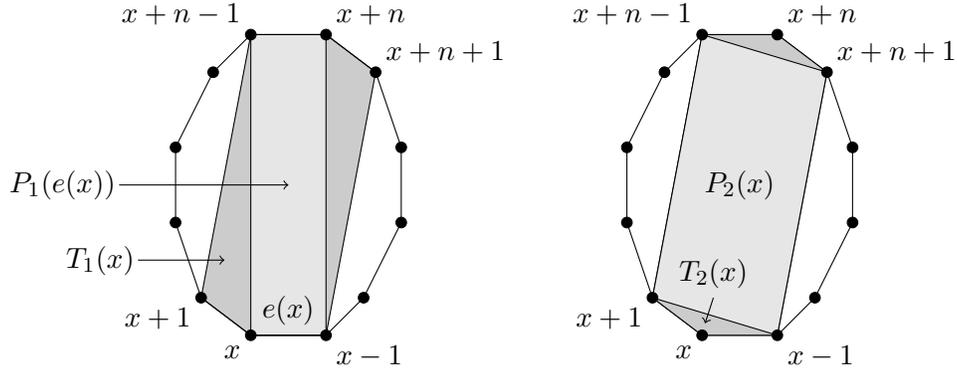
\begin{figure}[h]
\begin{center}
\begin{tikzpicture}
%\draw (0.5,-2) node[below right] {$x_i$} -- (-0.5,-2) node[below left] {$x_{i+1}$} -- (-1.16,-1.5) node[left] {$x_{i+2}$} -- (-1.5,-0.5) -- (-1.5,0.5) -- (-1,1.5) -- (-0.5,2) node[above left] {$x_{n+i}$} -- (0.5,2) node[above right] {$x_{n+i+1}$} -- (1.16,1.5) node[right] {$x_{n+i+2}$} -- (1.5,0.5) -- (1.5,-0.5) -- (1,-1.5) --cycle;
%
%\filldraw[fill=black!10!white] (0.5,-2) -- (-0.5,-2) -- (-1.16,-1.5) -- (-0.5,2) -- (0.5,2) -- (1.16,1.5) -- cycle;
%
%\filldraw (0.5,-2) circle (2pt);
%\filldraw (-0.5,-2) circle (2pt);
%\filldraw (-1.16,-1.5) circle (2pt);
%\filldraw (-1.5,-0.5) circle (2pt);
%\filldraw (-1.5,0.5) circle (2pt);
%\filldraw (-1,1.5) circle (2pt);
%\filldraw (-0.5,2) circle (2pt);
%\filldraw (0.5,2) circle (2pt);
%\filldraw (1.16,1.5) circle (2pt);
%\filldraw (1.5,0.5) circle (2pt);
%\filldraw (1.5,-0.5) circle (2pt);
%\filldraw (1,-1.5) circle (2pt);

\begin{scope}[shift={(4,0)}]
\filldraw[fill=black!20!white] (0.5,-2) -- (-0.5,-2) -- (-1.16,-1.5) -- (-0.5,2) -- (0.5,2) -- (1.16,1.5) -- cycle;
\filldraw[fill=black!10!white] (0.5,-2) -- (-0.5,-2) -- (-0.5,2) -- (0.5,2) -- cycle;

\draw (0.5,-2) node[below right] {$x-1$} --  node[above] {$e(x)$} (-0.5,-2) node[below left] {$x$} -- (-1.16,-1.5) node[below left] {$x+1$} -- (-1.5,-0.5) -- (-1.5,0.5) -- (-1,1.5) -- (-0.5,2) node[above left] {$x+n-1$} -- (0.5,2) node[above right] {$x+n$} -- (1.16,1.5) node[above right] {$x+n+1$} -- (1.5,0.5) -- (1.5,-0.5) -- (1,-1.5) -- cycle;

\filldraw (0.5,-2) circle (2pt);
\filldraw (-0.5,-2) circle (2pt);
\filldraw (-1.16,-1.5) circle (2pt);
\filldraw (-1.5,-0.5) circle (2pt);
\filldraw (-1.5,0.5) circle (2pt);
\filldraw (-1,1.5) circle (2pt);
\filldraw (-0.5,2) circle (2pt);
\filldraw (0.5,2) circle (2pt);
\filldraw (1.16,1.5) circle (2pt);
\filldraw (1.5,0.5) circle (2pt);
\filldraw (1.5,-0.5) circle (2pt);
\filldraw (1,-1.5) circle (2pt);

\draw[->] (-2,-1) -- (-0.85,-1);
\draw (-2.5,-1) node {$T_1(x)$};
\draw[->] (-2.25,0) -- (0,0);
\draw (-3,0) node{$P_1(e(x))$};
\end{scope}

\begin{scope}[shift={(10,0)}]
\draw (0.5,-2) node[below right] {$x-1$} --  (-0.5,-2) node[below left] {$x$} -- (-1.16,-1.5) node[below left] {$x+1$} -- (-1.5,-0.5) -- (-1.5,0.5) -- (-1,1.5) -- (-0.5,2) node[above left] {$x+n-1$} -- (0.5,2) node[above right] {$x+n$} -- (1.16,1.5) node[above right] {$x+n+1$} -- (1.5,0.5) -- (1.5,-0.5) -- (1,-1.5) -- cycle;

\filldraw[fill=black!20!white] (0.5,-2) -- (-0.5,-2) -- (-1.16,-1.5) -- (-0.5,2) -- (0.5,2) -- (1.16,1.5) -- cycle;
\filldraw[fill=black!10!white] (0.5,-2) -- (-1.16,-1.5) -- (-0.5,2) -- (1.16,1.5) -- cycle;

\filldraw (0.5,-2) circle (2pt);
\filldraw (-0.5,-2) circle (2pt);
\filldraw (-1.16,-1.5) circle (2pt);
\filldraw (-1.5,-0.5) circle (2pt);
\filldraw (-1.5,0.5) circle (2pt);
\filldraw (-1,1.5) circle (2pt);
\filldraw (-0.5,2) circle (2pt);
\filldraw (0.5,2) circle (2pt);
\filldraw (1.16,1.5) circle (2pt);
\filldraw (1.5,0.5) circle (2pt);
\filldraw (1.5,-0.5) circle (2pt);
\filldraw (1,-1.5) circle (2pt);

\draw[->] (-0.35,-1.5) node[above] {$T_2(x)$} -- (-0.46,-1.84);
\draw node {$P_2(x)$};
\end{scope}
\end{tikzpicture}
\caption{The two parallelograms from the proof of \cref{lem:quads.in.A} carved out of the same hexagon.}\label{fig:carve}
\end{center}
\end{figure}
The desired conclusion is therefore equivalent to the statement that
\[
\sum_{x\in V(A)}\area(T_1(x))>\sum_{x\in V(A)}\area(T_2(x)).
\]
However, since we are summing over all vertices $x$ we may replace $x$ by $x-1$ on the left-hand side, meaning that this is equivalent to the statement that
\[
\sum_{x\in V(A)}\area(T_1(x-1))>\sum_{x\in V(A)}\area(T_2(x)).
\]
In fact, we claim that this inequality holds not just in the sum, but term by term, in the sense that
\begin{equation}
\label{eq:Mahler.2d.exp.term}
\area(T_1(x-1))>\area(T_2(x))
\end{equation}
for every $x\in V(A)$. To see that this is true, note that the triangles $T_1(x-1)$ and $T_2(x)$ may both be thought of as having the edge $[x-1,x]$ as a base. Without loss of generality, we may assume that this edge is horizontal and that the body $A$ lies above it, as illustrated in Figure \ref{fig:triangles}.

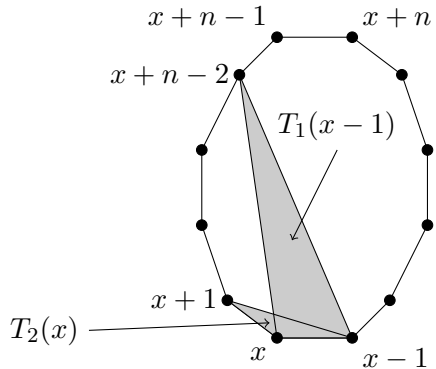
\begin{figure}[b]
\begin{center}
\begin{tikzpicture}

\draw (0.5,-2) node[below right] {$x-1$} -- (-0.5,-2) node[below left] {$x$} -- (-1.16,-1.5) node[left] {$x+1$} -- (-1.5,-0.5) -- (-1.5,0.5) -- (-1,1.5) node[left] {$x+n-2$} -- (-0.5,2) node[above left] {$x+n-1$} -- (0.5,2) node[above right] {$x+n$} -- (1.16,1.5) -- (1.5,0.5) -- (1.5,-0.5) -- (1,-1.5) --cycle;

\filldraw[fill=black!20!white]  (0.5,-2) -- (-0.5,-2) -- (-1,1.5) -- cycle;
\filldraw[fill=black!20!white] (0.5,-2) -- (-0.5,-2) -- (-1.16,-1.5) -- cycle;
\draw (-219/402,-678/402) -- (-0.5,-2);

\filldraw (0.5,-2) circle (2pt);
\filldraw (-0.5,-2) circle (2pt);
\filldraw (-1.16,-1.5) circle (2pt);
\filldraw (-1.5,-0.5) circle (2pt);
\filldraw (-1.5,0.5) circle (2pt);
\filldraw (-1,1.5) circle (2pt);
\filldraw (-0.5,2) circle (2pt);
\filldraw (0.5,2) circle (2pt);
\filldraw (1.16,1.5) circle (2pt);
\filldraw (1.5,0.5) circle (2pt);
\filldraw (1.5,-0.5) circle (2pt);
\filldraw (1,-1.5) circle (2pt);

\draw[->] (0.3,0.5) node[above] {$T_1(x-1)$} -- (-0.3,-0.7);
\draw[->] (-3,-1.9) node[left] {$T_2(x)$} -- (-0.6,-1.8);

\end{tikzpicture}
\caption{Illustration of the triangles appearing in the inequality \eqref{eq:Mahler.2d.exp.term}.}\label{fig:triangles}
\end{center}
\end{figure}

Now by symmetry the edge $[x+n-1,x+n]$ is horizontal, and so by convexity none of the edges lying on the clockwise path from $[x-1,x]$ to $[x+n-1,x+n]$ can be horizontal or downward sloping from $x+j$ to $x+j+1$. The vertical components of $x,x+1,\ldots,x+n-1$ are therefore strictly increasing. Since $n\ge4$, we may therefore conclude that $x+n-2$ has a vertical component strictly greater than that of $x+1$. The triangle $T_1(x-1)$ therefore has a greater height than $T_2(x)$ and the same base, as illustrated in Figure \ref{fig:triangles}, and so (\ref{eq:Mahler.2d.exp.term}) holds and the lemma is proved.
\end{proof}

The proof of \cref{prop:bures} is also relatively straightforward, but we make it easier to follow with two simple lemmas.

\begin{lemma}\label{lem:Mahler.2d.basic}Let $A$ be a centrally symmetric convex polygon in $\R^2$. Suppose that $s\in E(A)$ is horizontal, let $s'\in E(A)$ be an edge adjacent to $s$, and write $z$ for the intersection of the vertical axis and the line in $\R^2$ containing $s'$, as illustrated in Figure \ref{fig:Mahler.2d.basic}. Write $\alpha$ for the vertical distance from $z$ to $s$, write $\beta$ for the vertical distance from $v^\circ(s)$ to $v^\circ(s')$, and write $\gamma$ for the vertical distance from $v^\circ(s')$ to the origin. Then
\begin{displaymath}
\frac{\alpha}{c(s)}=\frac{\beta}{\gamma}.
\end{displaymath}
\end{lemma}

\begin{figure}[t]
\begin{center}
\begin{tikzpicture}
\draw[dotted] (0.75,3.5) -- (0,5) -- node[left] {$\alpha$} (0,3.5) -- node[left] {$c(s)$} (0,0) -- node[below right] {$c(s')$} (2,1);
\draw (-0.25,3.5) -- (-0.1,3.5) -- node[above] {$s$} (0.75,3.5) -- node[anchor=south west] {$s'$} (2.25,0.5);
\draw (1.9,1.2) -- (1.7,1.1) -- (1.8,0.9);
\draw (-0.05,5.05) -- (0.05,5-0.05);
\draw (-0.05,5-0.05) -- (0.05,5.05) node[above] {$z$};
\draw (-0.05,0.05) -- (0.05,-0.05);
\draw (0.05,0.05) -- (-0.05,-0.05) node[below] {\textup{0}};
\filldraw (0.75,3.5) circle (2pt);
\draw (0,0.5) arc (90:30:0.5);
\draw (70:0.5) node[above right] {$\theta$};
\draw (5-0.05,1.05) -- (5.05,1-0.05);
\draw (5.05,1.05) -- (5-0.05,1-0.05) node[below] {\textup{0}};
\filldraw (5,3.142857) circle (2pt) node[above right] {$v^\circ(s)$};
\filldraw (8,2.5) circle (2pt) node[above right] {$v^\circ(s')$};
\draw (5,3.142857) -- (8,2.5);
\draw[dotted] (5,3.142857) -- node[left] {$\beta$} (5,2.5) -- node[left] {$\gamma$} (5,1) -- (8,2.5) -- (4.75,2.5);
\draw (5,1.5) arc (90:30:0.5);
\draw (5,1)+(70:0.5) node[above right] {$\theta$};
\end{tikzpicture}
\caption{Illustration of the hypotheses of \cref{lem:Mahler.2d.basic}.}\label{fig:Mahler.2d.basic}
\end{center}
\end{figure}
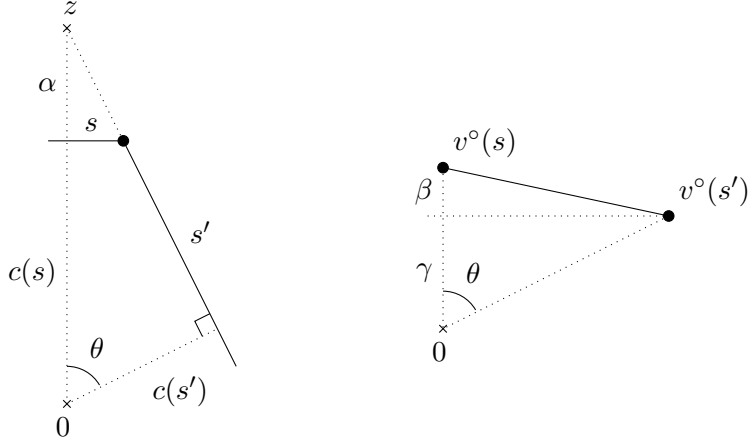

\begin{proof}
We have $\|v^\circ(s')\|=1/c(s')$ by (\ref{eq:dual.def}), and by the equality of the angles labelled $\theta$ in Figure \ref{fig:Mahler.2d.basic} we may compare similar triangles to conclude that $\|v^\circ(s')\|/\gamma=(\alpha+c(s))/c(s')$. Combining these gives $\alpha+c(s)=1/\gamma$. Another instance of (\ref{eq:dual.def}) implies that $\beta+\gamma=1/c(s)$. These last two equalities give $\alpha\gamma=1-c(s)\gamma=\beta c(s)$, which gives the desired result.
\end{proof}

An immediate consequence of \cref{lem:Mahler.2d.basic} is the following fact.

\begin{lemma}\label{cor:Mahler.2d.basic}
Let $A$ be a centrally symmetric convex polygon in $\R^2$, and suppose that $s\in E(A)$ is horizontal and that the lines containing the edges adjacent to the edge $s$ intersect on the vertical axis. Then the vertices adjacent to $v^\circ(s)$ lie on a common horizontal line.
\end{lemma}

\begin{proof}[Proof of \cref{prop:bures}]

It follows from \eqref{eq:lin.tr} that applying an invertible linear transormation to $A$ not affect the value of either side of the inequality (\ref{eq:Mahler.2d.j}). We are therefore free to apply such a transformation to $A$. In particular, by rotating we may assume that $s$ is horizontal, and by applying a suitable shear we may assume that the intersection $z$ of the lines containing the edges adjacent to $s$ lies on the vertical axis, as illustrated in Figure \ref{fig:bures}. Note that the edge $[v^\circ(s)-1,v^\circ(s)+1]$ of $P_2(v^\circ(s))$ is also horizontal by \cref{cor:Mahler.2d.basic}. Having applied these transformations, we write $\beta$ for the vertical distance from $v^\circ(s)$ to $P_2(v^\circ(s))$, and $\gamma$ for the distance from $v^\circ(s)$ to the origin minus $\beta$, as in \cref{lem:Mahler.2d.basic}. This notation is also recorded in Figure \ref{fig:bures}.

\begin{figure}[t]
\begin{center}
\begin{tikzpicture}
\draw (-3,1.5) -- (-2,2) -- node[above] {$s$} (1,2) -- (2.5,0.5) -- (3,-1.5);
\draw (3,-1.5) -- (2,-2) -- node[below] {$-s$} (-1,-2) -- (-2.5,-0.5) -- (-3,1.5);

\draw (0,-4) node {$A$};
\draw (7,-4) node {$A^\circ$};

\filldraw (-3,1.5) circle (2pt);
\filldraw (-2,2) circle (2pt);
\filldraw (1,2) circle (2pt);
\filldraw (2.5,0.5) circle (2pt);
\filldraw (3,-1.5) circle (2pt);
\filldraw (2,-2) circle (2pt);
\filldraw (-1,-2) circle (2pt);
\filldraw (-2.5,-0.5) circle (2pt);

\draw[dotted] (-2,2) -- (0.5,3.25);
\draw[dotted] (1,2) -- (-0.25,3.25);

\draw[dotted] (2,-2) -- (-0.5,-3.25);
\draw[dotted] (-1,-2) -- (0.25,-3.25);

\draw (-2,2) -- (-1,-2);
\draw (2,-2) -- (1,2);

\draw[dotted] (0,3.5) -- (0,-3.5);

\draw (-0.75,0.75) node {$P_1(s)$};
\draw (-0.05,3.05) node[left] {$z$} -- (0.05,3-0.05);
\draw (-0.05,3-0.05) -- (0.05,3.05);
\draw (0.05,-3.05) -- (-0.05,-3+0.05) node[left] {$-z$};
\draw (0.05,-3+0.05) -- (-0.05,-3.05);
\draw (-0.05,0.05) -- (0.05,-0.05);
\draw (-0.05,-0.05) -- (0.05,0.05) node[right] {$0$} ;
\draw (6,2) node[above left] {$v^\circ(s)-1$} -- (7,3) node[above left] {$v^\circ(s)$} -- (9,2) node[above right] {$v^\circ(s)+1$} -- (9.15,0.75) -- (8,-2) -- (7,-3) node[below right] {$-v^\circ(s)$} -- (5,-2) -- (4.85,-0.75) -- cycle;
\draw (6,2) -- (9,2) -- (8,-2) -- (5,-2) -- cycle;

\draw[dotted] (7,3.5) -- (7,3);
\draw[dotted] (7,0) -- (7,-3.5);
\draw (7.2,2.4) node {$\beta$};
\draw (7.2,1) node {$\gamma$};
\draw[<->] (7,2.925) -- (7,2);
\draw[<->] (7,0) -- (7,2);

\filldraw (6,2) circle (2pt);
\filldraw (7,3) circle (2pt);
\filldraw (9,2) circle (2pt);
\filldraw (8,-2) circle (2pt);
\filldraw (7,-3) circle (2pt);
\filldraw (5,-2) circle (2pt);
\filldraw (4.85,-0.75) circle (2pt);
\filldraw (9.15,0.75) circle (2pt);

\draw (7-0.05,0.05) -- (7.05,-0.05);
\draw (7-0.05,-0.05) -- (7.05,0.05) node[right] {$0$} ;

\draw (6.25,-0.75) node {$P_2(v^\circ(s))$};
\end{tikzpicture}
\caption{Illustration of the proof of \cref{prop:bures}.}\label{fig:bures}
\end{center}
\end{figure}
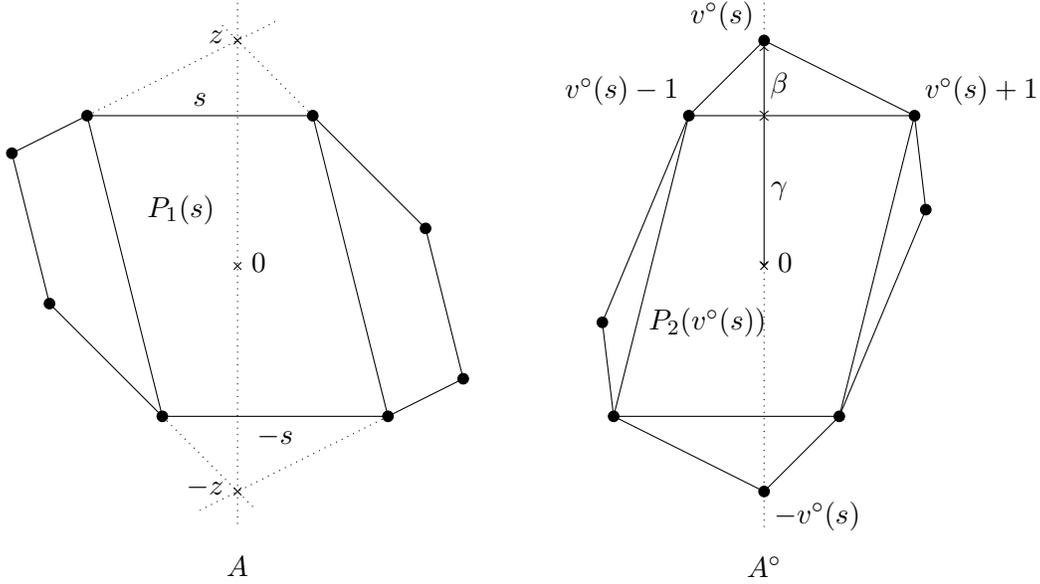

The body $A_s$ is obtained from the body $A$ by removing the edges $\pm s$ and extending their neighbouring edges to the points $\pm z$, which will form a pair of vertices of $A_s$. The body $(A_s)^\circ$ is obtained from the body $A^\circ$ by removing the vertices $\pm v^\circ(s)$ and taking the convex hull of the remaining vertices. Again, the reader may find it helpful to refer to Figure \ref{fig:bures}.

This all implies that
\[
\area((A_s)^\circ)=\area(A^\circ)-\frac{\beta}{2\gamma}\area(P_2(v^\circ(s))),
\]
and combined with \cref{lem:Mahler.2d.basic} implies that
\[
\area(A_s)=\area(A)+\frac{\beta}{2\gamma}\area(P_1(s)).
\]
Multiplying these equations together, we see that
\[
\begin{split}
M(A_s)=M(A)+\frac{\beta}{2\gamma}\Big(\area(A^\circ)\area(P_1(s))-\area(A)\area(P_2(v^\circ(s)))\Big)
        \\-\frac{\beta^2}{4\gamma^2}\area(P_1(s))\area(P_2(v^\circ(s))).
\end{split}
\]
Since the last term of this expression is always negative, the proposition follows from (\ref{eq:Mahler.2d.j}).
\end{proof}

\end{document}